\documentclass[amsfonts,reqno,12pt]{amsart}

\usepackage{epsfig}
\usepackage{amssymb}
\usepackage{xcolor}
\newtheorem{thm}{Theorem}[section]
\newtheorem{deff}[thm]{Definition}
\newtheorem{lem}[thm]{Lemma}
\newtheorem{cor}[thm]{Corollary}
\newtheorem{rem}[thm]{Remark}

\newtheorem{prob}[thm]{Problem}

\newcommand{\R}{{\Bbb R}}

\begin{document}

 \title[Non-central sections of convex bodies]{Non-central sections of convex bodies}

 \author{V.~Yaskin and N.~Zhang}

  \address{Vladyslav Yaskin, Department of Mathematical and Statistical Sciences, University of Alberta, Edmonton, Alberta, T6G 2G1, Canada}
 \email{yaskin@ualberta.ca}
 \address{Ning Zhang, Department of Mathematical and Statistical Sciences, University of Alberta, Edmonton, Alberta, T6G 2G1, Canada}
 \email{nzhang2@ualberta.ca}



\thanks{Both authors are partially supported by a grant from NSERC}

 \begin{abstract} We study the following open problem, suggested by Barker and Larman. Let $K$ and $L$ be convex bodies in $\mathbb R^n$ ($n\ge 2$) that contain a Euclidean ball $B$ in their interiors. If $\mathrm{vol}_{n-1}(K\cap H) = \mathrm{vol}_{n-1}(L\cap H)$ for every hyperplane $H$ that supports $B$, does it follow that $K=L$? We discuss various modifications of this problem. In particular, we show that in $\mathbb R^2$ the answer is positive if the above condition is true for two disks, none of which is contained in the other. We also study some higher dimensional analogues.

\end{abstract}

\maketitle

\section{Introduction}
Geometric Tomography is an area of Mathematics that deals with the study of properties of objects (such as convex bodies or star bodies) based on information about the size of their sections, projections, etc.  It is a well-known result, which goes back to Minkowski and Funk (see \cite{Ga}), that an origin-symmetric star body in $\mathbb R^n$ is uniquely determined by the areas of its central sections. More precisely, if $K$ and $L$ are origin-symmetric star bodies in $\mathbb R^n$ such that $$\mathrm{vol}_{n-1} (K\cap H) = \mathrm{vol}_{n-1} (L\cap H)$$
for every hyperplane $H$ passing through the origin, then $K=L$. On the other hand, in the class of general (not necessarily symmetric) star bodies the latter result is not true.

In view of this, it is natural to ask what information is needed to determine non-symmetric bodies. Falconer \cite{F} and Gardner \cite{Ga1} have shown that if $K$ and $L$ are convex bodies in $\mathbb R^n$ that contain two points $p$ and $q$ in their interiors and such that $\mathrm{vol}_{n-1} (K\cap H) = \mathrm{vol}_{n-1} (L\cap H)$ for every hyperplane $H$ that passes through either  $p$ or $q$, then $K=L$. In this context, let us also mention the problem of Klee about  the inner section function of convex bodies, which is given by $m_K(u) = \max_{t\in \mathbb R} \mathrm{vol}_{n-1} (K\cap \{u^\perp+tu\})$. In 1969 Klee asked whether the knowledge of $m_K$ is sufficient to determine the body $K$ uniquely. In \cite{GRYZ} the problem was solved in the negative, and a little later a nonspherical body with a constant inner section function was constructed in \cite{NRZ}.

Recently, a lot of attention has been attracted to the following problem, posed by Barker and Larman in \cite{BL}. Note that  a similar question  on the sphere was considered earlier by Santal\'{o} \cite{Sa}.

 \begin{prob}\label{BL1} Let $K$ and $L$ be convex bodies in $\mathbb R^n$ ($n\ge 2$) that contain a Euclidean ball $B$ in their interiors. If $\mathrm{vol}_{n-1}(K\cap H) = \mathrm{vol}_{n-1}(L\cap H)$ for every hyperplane $H$ that supports $B$, does it follow that $K=L$?
 \end{prob}
The problem is open even in $\mathbb R^2$. Some particular cases are known to be true. In particular,   a body $K$ in $\mathbb R^2$ all of whose sections by lines supporting a disk have the same length,   must itself be a disk; see \cite{BL}. The problem also has a positive answer in the class of convex polytopes in $\mathbb R^n$; see \cite{Y}.

Barker and Larman also suggested a more general version of Problem \ref{BL1}.

 \begin{prob}\label{BL2} Let $K$ and $L$ be convex bodies in $\mathbb R^n$ ($n\ge 2$) that contain a convex body $D$ in their interiors. If $\mathrm{vol}_{n-1}(K\cap H) = \mathrm{vol}_{n-1}(L\cap H)$ for every hyperplane $H$ that supports $D$, does it follow that $K=L$?
 \end{prob}

In this paper we study the following  modification of  Problem \ref{BL2}.

 \begin{prob}\label{YZ1} Let $K$ and $L$ be convex bodies in $\mathbb R^n$ ($n\ge 2$) that contain two convex bodies $D_1$ and $D_2$ in their interiors. If $\mathrm{vol}_{n-1}(K\cap H) = \mathrm{vol}_{n-1}(L\cap H)$ for every hyperplane $H$ that supports either $D_1$ or $D_2$, does it follow that $K=L$?
 \end{prob}
We show that the problem has a positive answer in $\mathbb R^2$ under some mild assumptions on $D_1$ and $D_2$.
We also discuss the following closely related problem.

\begin{prob}\label{YZ2}
Let $K$ an $L$ be convex bodies in $\mathbb R^n$ and let $D$ be a convex body in   the interior of $K\cap L$.     If $ \textnormal{vol}_n(K\cap H^+) =  \textnormal{vol}_n(L\cap H^+)$ for every hyperplane $H$ supporting $D$, does it follow that $K=L$? Here, $H^+$ is the half-space bounded by the hyperplane $H$ that does not intersect the interior of $D$.
\end{prob}

Again, we solve a two-dimensional modification of this problem by taking two bodies $D_1$ and $D_2$ in  the interior of $K\cap L$.

We also discuss some higher-dimensional analogues. In particular, Groemer \cite{Gr} has shown that convex bodies are uniquely determined by the areas of ``half-sections". More precisely, consider half-planes of the form $H(u,w) = \{x\in \mathbb R^n: x\in u^\perp, \langle x, w\rangle \ge 0\}$, where $u\in S^{n-1}$ and $w\in S^{n-1}\cap u^\perp$. Then the equality $\mbox{vol}_{n-1}(K\cap H(u,w)) = \mbox{vol}_{n-1}(L\cap H(u,w))$ for all such half-planes implies that $K=L$.
We give a version of this result for half-planes that do not pass through the origin. Some other types of sections are also discussed.

\section{Definitions  and preliminaries}

In this section we collect some basic concepts and definitions that we use in the paper. For further facts in Convex Geometry and Geometric Tomography the reader is referred to the books by Schneider \cite{Sch} and Gardner \cite{Ga}.

 A set  in $\mathbb{R}^n$ is called {\it convex} if it contains the closed line segment joining any two of its points. A convex set is a {\it convex body} if it is compact and has non-empty interior.  A convex body is {\it strictly convex} if its boundary contains no line segments.

A hyperplane $H$ {\it supports} a set $E$ at a point $x$ if $x\in E\cap H$ and $E$ is contained in one of the two closed half-spaces bounded by $H$.
We say $H$ is a {\it supporting hyperplane} of $E$ if $H$ supports $E$ at some point.

The {\it support function} of $K$ is defined by
$$
h_K(x)=\max\{\langle x, y\rangle: y\in K\},
$$
for $x\in \mathbb{R}^n$.
 If $h_K$ is of class $C^k$ on $\R^n\backslash \{O\}$, we will simply say that  $K$ has  a  $C^k$ support function. For a convex body $K\subset \mathbb R^2$ it is often convenient to write $h_K$ as a function of the polar angle $\theta$. So, abusing notation, we will use $h_K(\theta)$ to denote $h_K((\cos\theta,\sin\theta))$. If $H$ is the supporting line to $K\subset \mathbb R^2$ with the outer normal vector $(\cos\theta,\sin\theta)$, and $K$ has a $C^1$ support function,  then $K$ has a unique point of contact with $H$, and $|h_K^{'}(\theta)|$ is the distance from this point  to the foot of the perpendicular from the origin $O$ to $H$; see \cite[p.~24]{Ga}.

A compact set $L$ is called a {\it star body} if the origin $O$ is an interior point of $L$, every line through $O$ meets $L$ in a line segment, and its {\it Minkowski functional} defined by
$$\|x\|_L = \min\{a\ge 0: x \in aL\}$$ is a continuous function on $\mathbb R^n$.

The {\it radial function} of $L$ is given by $\rho_L (x) = \|x\|_L^{-1}$, for $x\in \mathbb{R}^n\backslash \{O\}$. If $x\in S^{n-1}$, then $\rho_L(x)$ is just the radius of $L$ in the direction of $x$.
If $p$ is a point in the interior of $L$, and $L- p$ is a star body, then we will use $\rho_{L,p}$ to denote $\rho_{L-p}$.

Let
$K$ be a convex body in $\mathbb{R}^n$, and $D$ be a strictly convex body in the interior of $K$. Let $H$ be a supporting plane to $D$ with outer unit normal $\xi$, and $p=D\cap H$ be the corresponding point of contact.  If $u\in S^{n-1}\cap \xi^\perp$, we denote by  $\rho_{K,D}(u,\xi) = \rho_{K,p}(u)$   the radial function of $K\cap H$ with respect to $p$.

{Let $\mathcal S(\mathbb R^n)$ be the Schwartz space of   infinitely differentiable
rapidly decreasing functions on $\mathbb R^n$. Functions from this space are called test functions. For a function $\psi\in \mathcal S(\mathbb{R}^n)$,   its {\it Fourier transform} is defined by
$$
\hat{\psi}(\xi)=\int_{\mathbb{R}^n}\psi(x)e^{-i\langle x,\xi\rangle}\,dx,\ \xi\in \mathbb{R}^n.
$$
\\
By $\mathcal S'(\mathbb R^n)$ we denote the space of continuous linear functionals on $\mathcal S(\mathbb R^n)$. Elements of this space are referred to as distributions. By $\langle f, \psi\rangle$ we denote the action of the distribution $f$ on the test function $\psi$. Note that $\hat\psi$ is also a test function, which allows to introduce the following definition. We say that the distribution $\hat f$  is {\it the Fourier transform of the distribution $f$}
if
$$
\langle \hat{f},\psi\rangle=\langle f,\hat{\psi}\rangle,
$$
for every test function $\psi$.
The reader is referred to the book \cite{K} for applications of Fourier transforms to the study of convex bodies.
}

\section{Main results: 2-dimensional cases.}

 We will start with the following definition. We say that convex bodies $D_1$ and $D_2$ in $\mathbb R^2$ are {\it admissible} if they have $C^2$ support functions, $D_1\cup D_2$ is not convex, and there are only two   lines that support both $D_1$ and $D_2$ and do not separate $D_1$ and $D_2$. The last condition is satisfied, when, for example, the bodies $D_1$ and $D_2$ are disjoint, or they touch each other, or they overlap, but their boundaries have only two common points.

\begin{center}
\includegraphics[scale=0.40]{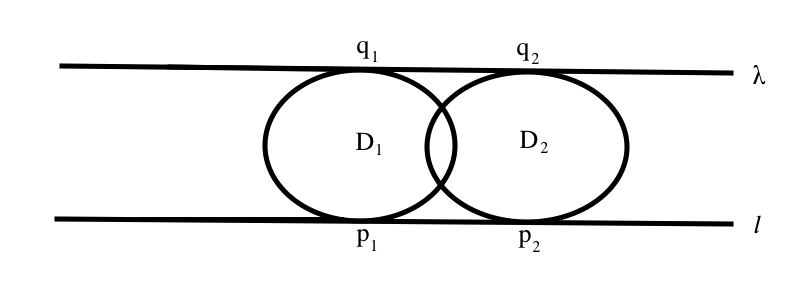}\\ Figure 1
\end{center}
\begin{center}
\includegraphics[scale=0.40]{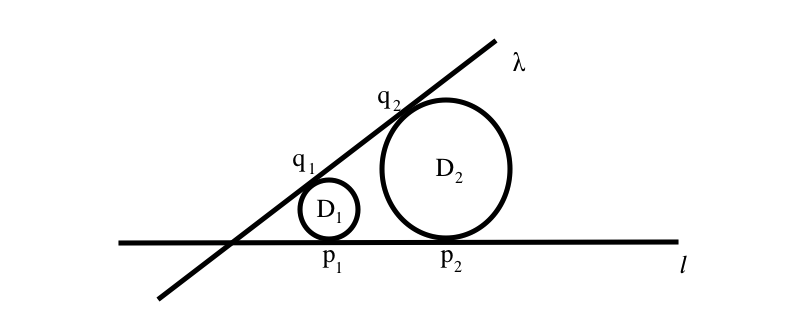}\\
   Figure 2
\end{center}

 Figures 1 and 2 show two examples of  admissible convex bodies.
 For simplicity, the reader could just think of two disks (not necessarily of the same radius) such that none of them is contained in the other.

The following is one of the main results of this section.

\begin{thm}\label{main}
Let $K$ and $L$ be convex bodies in $\mathbb R^2$ and {let $D_1$ and $D_2$ be two admissible convex bodies in the interior of $K\cap L$.} If the chords $K\cap H$ and $L\cap H$ have equal length for all $H$ supporting either $D_1$ or $D_2$, then $K=L$.
\end{thm}

We will obtain this theorem as a particular case of a more general statement, Theorem \ref{p3.4} below. First, we will need the following lemma.

\begin{lem}\label{l2.2}
Let $D\subset \mathbb{R}^2$ be a   convex body with a $C^2$ support function. {Let $Q\in \partial D$ and $l$ be the supporting line to $D$ at $Q$. Suppose the origin $O$ is located on the line perpendicular to $l$ and passing through $Q$, and $O\ne Q$.   Consider a polar coordinate system centered at $O$ with the polar axis $\overrightarrow{OQ}$.} Then, for $\theta$ small enough, we have
\begin{equation}\label{e2.1}
h'_D(\theta)\sin \theta+ h_D(0)-h_D(\theta)\cos \theta\approx \sin^2 \theta ,
\end{equation}
where $f\approx g$ means there exist two constants {$C_1,C_2$}, such that, $C_1 g\leq f\leq C_2 g$.
\end{lem}
\begin{proof}
Since $Q$ is both the point where $l$ supports $D$ and the foot of the perpendicular from $O$ to $l$, it follows that  $h'_D(0)=0$. Thus,
$$
h_D(\theta)=h_D(0) +\frac{h''_D(0)}{2}\theta^2+o(\theta^2).
$$
Therefore,   for $\theta$ small enough,  we have
\begin{eqnarray*}
&&
h_D(0)-h_D(\theta)\cos \theta\\
&&=h_D(0) - \left(h_D(0) +\frac{h''_D(0)}{2}\theta^2+o(\theta^2)\right)\left(1 -\frac12 \theta^2 + o(\theta^2) \right)\\
&& = \frac{h_D(0)-h''_D(0)}{2}\theta^2+ o(\theta^2)\\
&& \approx \sin^2\theta,
\end{eqnarray*}
and
$$
h'_D(\theta)\sin \theta = (h''_D(0)\theta+o(\theta))(\theta + o(\theta))\approx \sin^2\theta.
$$
\end{proof}

\begin{thm}\label{p3.4}
Let $K$ and $L$ be convex bodies in $\mathbb R^2$ and let $D_1$ and $D_2$ be two admissible convex bodies in the interior of $K\cap L$. Assume that for some $i>0$ one of the following two conditions holds:
\begin{enumerate}
\item[(I)] $\rho_{K,D_j}^i(u,\xi)+\rho_{K,D_j}^i(-u,\xi)=\rho_{L,D_j}^i(u,\xi)+\rho_{L,D_j}^i(-u,\xi)$, for $j=1,2$,
\item[(II)] $\partial K\cap \partial L\ne \emptyset$ and $\rho_{K,D_j}^i(u,\xi)-\rho_{K,D_j}^i(-u,\xi)=\rho_{L,D_j}^i(u,\xi)-\rho_{L,D_j}^i(-u,\xi)$, for $j=1,2$,
\end{enumerate}
for all $ \xi, u \in S^1$ such that $u\perp \xi$.

{Then $K=L$.}
\end{thm}

\begin{proof}
{
We will present the proof of the theorem only using condition (I). The other case is similar and we will just make a brief comment on how the proof should be adjusted.
}

{\it Step 1.}
 {Since there are two common supporting lines to $D_1$ and $D_2$ (that do not separate $D_1$ and $D_2$), we will denote them by $l$ and $\lambda$, and let $p_1 = D_1\cap l$, $q_1 = D_1\cap \lambda$, $p_2 = D_2\cap l$, $q_2 = D_2\cap \lambda$; see Figures 1 and 2.  We claim that at least one of the (possibly degenerate) segments $[p_1,p_2]$ or $[q_1,q_2]$ is not entirely contained   in $D_1\cup D_2$. We will prove this claim in a slightly more general setting, i.e. without   the assumption that $D_1$ and $D_2$ are strictly convex. In that case, instead of single points of contact we may have intervals, and $[p_1,p_2]$ or $[q_1,q_2]$ will just stand for the convex hulls of  the corresponding support  sets. To prove the claim, we will argue by contradiction. Assume that $[p_1,p_2]$ and $[q_1,q_2]$ are contained in $D_1\cup D_2$. Then there are points $p\in [p_1,p_2]$ and $q\in [q_1,q_2]$ that both belong to $D_1\cap D_2$. We can assume that the origin is an interior point of the interval $[p,q]$. Since there are only two common supporting lines to $D_1$ and $D_2$, we have exactly two directions $u_1$ and $u_2$, such that $h_{D_1}(u_1) = h_{D_2}(u_1)$ and $h_{D_1}(u_2) = h_{D_2}(u_2)$. These directions divide the circle $S^1$ into two open arcs $U_1$ and $U_2$, satisfying $h_{D_1}(u) > h_{D_2}(u)$ for all $u\in U_1$,  and $h_{D_1}(u) < h_{D_2}(u)$ for all $u\in U_2$. Thus the line $l(p,q)$ through the points $p$ and $q$ cuts each of the bodies $D_1$ and $D_2$ into two convex parts: $D_1 = D_{11} \cup D_{12}$ and $D_2 = D_{21} \cup D_{22}$, such that $D_{11} \supset D_{21}$ and $D_{12} \subset D_{22}$. In other words, $D_1\cup D_2 = D_{11}\cup D_{22}$, where $D_{11}$ and $D_{22}$ are separated by $l(p,q)$. Now, if we take two points $X$, $Y\in D_1\cup D_2$, then we have two cases: either they lie on one side of $l(p,q)$, or on different sides. In the first case, either $X$, $Y\in D_{11}$, or $X$, $Y\in D_{22}$, which means that $[X,Y]\subset D_1\cup D_2$. In the second case, the segment $[X,Y]$ intersects $[p,q]$, and thus one part of $[X,Y]$ lies in $D_{11}$, and the other in $D_{22}$, which again implies that $[X,Y]\subset D_1\cup D_2$, meaning that $D_1\cup D_2$ is convex. Contradiction. Thus, we have proved  that at least one of the segments   $[p_1,p_2]$ or $[q_1,q_2]$ is not entirely contained   in $D_1\cup D_2$. We will assume it is  the segment $[p_1,p_2]$ and will fix the corresponding supporting line  $l$.

 \begin{center}
\includegraphics[width=0.7\linewidth]{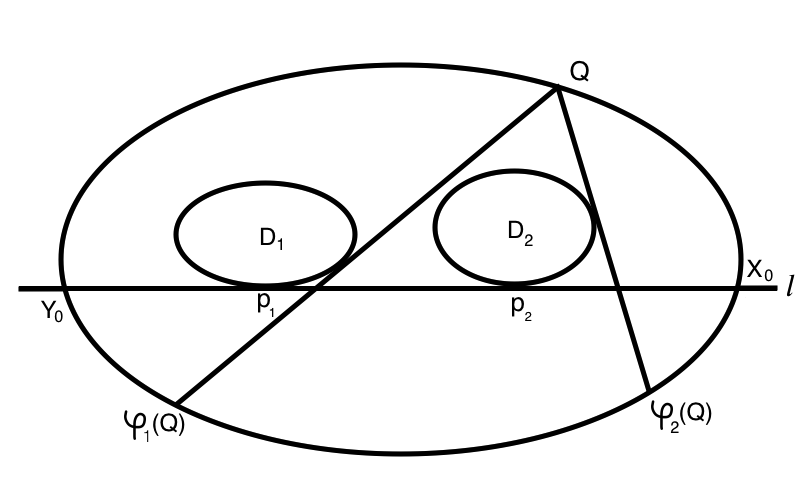}\\ Figure 3
\end{center}

{\it Step 2.} Here we will show that $\partial K \cap l = \partial L\cap l$. To this end, we define two mappings $\varphi_1$ and $\varphi_2$ (see Figure 3). We will start with  $\varphi_1$; the other is similar. Let $Q$ be a point outside of $D_1$. There are two unique supporting lines to $D_1$ passing through $Q$. {Choose the  one that lies on the left of the body $D_1$, when viewing from the point $Q$.} Let $T$ be the point of contact of the chosen supporting line and the body $D_1$. On this line we take a point $\varphi_1(Q)$, such that $T$ is inside the segment $[Q,\varphi_1(Q)]$ and
$$|QT|^i+|\varphi_1(Q)T|^i =\rho_{K,D_1}^i(u,\xi)+\rho_{K,D_1}^i(-u,\xi),$$
where $u$ is a unit vector parallel to $\overrightarrow{T Q}$ and $\xi$ is the outward unit normal vector to $D_1$  at $T$ (which is perpendicular to $u$).    The definition for $\varphi_2$ is similar; one only needs to replace $D_1$ by $D_2$. Note that the domains of $\varphi_1$ and $\varphi_2$ include the symmetric difference $K\triangle L$. {An important observation is that if $Q$ is on the boundary of $K$ (resp. $L$), then $\varphi_1(Q)$, $\varphi_1^{-1}(Q)$, $\varphi_2(Q)$, and $\varphi_2^{-1}(Q)$ are also on the boundary of $K$ (resp. $L$).

Note that there exists at least one point $Q\in \partial K\cap \partial L$. Otherwise, one of $\partial K$ or $\partial L$ would be strictly contained inside the other, thus violating condition (1) of the proposition.
The line $l$ divides the plane into two closed half-planes $l^+$ and $l^-$, where $l^+$ is the one that contains $D_1$ and $D_2$. If $Q\in l^+$, then applying $\varphi_1$ finitely many times, we will get a point in $l^-$ (since $\varphi_1$ cannot miss the whole half-plane), which is also a common point of the boundaries of $K$ and $L$. Thus from now on we will assume that $Q\in l^-$.
If $Q\in l$, then the proof of Step 2 is finished. If $Q$ is strictly below $l$, we will apply the following procedure.

Without loss of generality, we can assume that, if the line $\lambda$ intersects $l$, then the point of intersection lies to the left of the point $p_1$, as in Figure 2. Let us also denote  by $X_0$ and $Y_0$  the points of intersection of the boundary of $K$ with the line $l$, as in Figure 3.
Let $Q_0 = \varphi_2^{-1}(Q)$. The line $l(Q,Q_0)$ through $Q$ and $Q_0$ is tangent to $D_2$ and therefore cannot have common points with $D_1$ (otherwise rolling this line along the boundary of $D_2$ we would find a third common supporting line to both $D_1$ and $D_2$). Now consider $\varphi_1(Q_0)$ and the line $l(\varphi_1(Q_0),Q_0)$ through $\varphi_1(Q_0)$ and $Q_0$. Note that $\varphi_1(Q_0)$ is below $l$. Since $l(Q,Q_0)$ and $l(\varphi_1(Q_0),Q_0)$ are different, the points $Q$ and $\varphi_1(Q_0)$ are also different. 
Moreover, we have $\angle(\overrightarrow{ \varphi_1(Q_0)Q_0},\overrightarrow{p_1 X_0})<\angle(\overrightarrow{ QQ_0},\overrightarrow{p_1 X_0})$. Repeating this procedure, we construct $Q_1 = \varphi_2^{-1}(\varphi_1(Q_0))$ and observe that $ \angle(\overrightarrow{ \varphi_1(Q_0)Q_1},\overrightarrow{p_1 X_0})< \angle(\overrightarrow{ \varphi_1(Q_0)Q_0},\overrightarrow{p_1 X_0})$. Continuing in this manner, we obtain a sequence of points $\{Q_j\}_{j=0}^\infty$ and a corresponding sequence of angles $\{\theta_j\}_{j=0}^\infty$, defined by $Q_{j+1}=\varphi^{-1}_2(\varphi_1(Q_j))$ and $\theta_j=\angle(\overrightarrow{ \varphi_1(Q_j)Q_j},\overrightarrow{p_1 X_0})$. We note that $Q_j \in l^+\cap \partial K\cap \partial L$, and $\theta_j > \theta_{j+1}$,  for all $j$. Thus, the sequence $\{\theta_j\}$ is strictly decreasing and positive, and therefore convergent. To reach a contradiction, let us assume that the limit is not zero. Then there is a point $\tilde Q = \lim_{j\to \infty} Q_j$ that lies above the line $l$ and satisfies $\varphi_1(\tilde Q) = \varphi_2(\tilde Q) $. Thus, we have a third line that supports both $D_1$ and $D_2$. Contradiction. Hence, $\lim_{j\to \infty}\theta_j=0$, and we conclude that $\partial K\cap l=\partial L\cap l=\{X_0,Y_0\}$.
}

{{\it Step 3.}   We will prove that $\partial K$ and $\partial L$ coincide in some one-sided neighborhood of the  point $X_0$. Since  $$\frac{|Y_0p_1||X_0p_2|}{|X_0p_1||Y_0p_2|}<1,$$ we can choose positive numbers $a,b,c, d$ such that
$$
0<a<|X_0p_1|,\ |Y_0p_1|<b,\  0<c<|Y_0p_2|,\ |X_0p_2|<d, \mbox{ and }
\frac{bd}{ac}<1.
$$
}

By the continuity of the boundaries of $K$, $L$, $D_1$, and $D_2$, there exist neighborhoods, $\mathcal{N}(X_0), \ \mathcal{N}(Y_0)$, of $X_0$ and $Y_0$ respectively, such that
\begin{equation}\label{abcd}
\begin{cases}
|XT_1|>a\mbox{ and }|XT_2|<d, & \mbox{if } X\in \mathcal{N}(X_0),\\
|YT_3|>c\mbox{ and }|YT_4|<b, & \mbox{if } Y\in \mathcal{N}(Y_0),
\end{cases}
\end{equation}
where $T_1$   is the point of intersection of $l$ and the line through $X$   supporting $D_1$ (if $X$ is itself on the line $l$, then we let $T_1 = p_1$).  Similarly, $T_2$  is the point of intersection of $l$ and the line through $X$   supporting $D_2$ (again, if $X$ is on the line $l$, then we let $T_2 = p_2$). Here and below, by the supporting lines we mean those that are closest to $l$. There is no ambiguity, since $X$ is sufficiently close to $l$. (The points $T_3$ and $T_4$ are defined similarly, if we replace $X$ by $Y$).

{Next we claim that there are points of $ \partial K\cap \partial L$ in the set $ \mathcal{N}(X_0)\cap l^+$. Indeed, if in Step 2 there was a point $Q\in \partial K\cap \partial L$ strictly below the line $l$, then the points from the corresponding sequence $\{Q_i\}$ all lie in $\partial K\cap \partial L \cap \mathcal{N}(X_0)\cap l^+$ for $i$ large enough. If in Step 2 the point $Q$ was on the line $l$, then we can take $\varphi_1(\varphi_2^{-1}(X_0))$, which will be strictly below $l$, and repeat the same procedure.  }

Our goal is to show that $ \partial K$ and $ \partial L$ coincide in $ \mathcal{N}(X_0)\cap l^+$. Taking a smaller neighborhood $\mathcal{N}(X_0)$ if needed, we can assume that $\varphi_1( \mathcal{N}(X_0)\cap l^+)\subset \mathcal{N}(Y_0)$.
Discarding finitely many terms of the sequence  $\{Q_j\}$, we can also assume that $Q_j \in \mathcal{N}(X_0)\cap l^+$ for all $j\ge 0$. Now consider the segments of the boundaries of $\partial K$ and $\partial L$ between the points $Q_0$  and $Q_1$. If they coincide, then we are done, since the  boundaries of $\partial K$ and $\partial L$ would have to coincide between $Q_j$ and $Q_{j+1}$ for all $j$. So, we will next assume that $\partial K$ and $\partial L$ are not identically the same between $Q_0$  and $Q_1$. Let $E_0$ be the  component of $  K\triangle L  $ with endpoints $Q_0$  and $Q_1$, i.e. $E_0$ is the subset of $( K\triangle L )\cap l^+$  located  between the lines $l(Q_0,\varphi_1(Q_0))$ and $l(Q_1,\varphi_1(Q_0))$. We will define a sequence of sets $\{E_j\}_{j=0}^\infty$, where $E_{j+1} =  \varphi_2^{-1}(\varphi_1(E_j))$. Each $E_j$ is a component of $  K\triangle L  $ with endpoints $Q_j$ and $Q_{j+1}$.

{Now consider a Cartesian coordinate system with $l$ being the $x$-axis, and the $y$-axis perpendicular to $l$. We will be using ideas similar to those in \cite[Section 5.2]{Ga}. For a measurable set $E$  define}
\begin{equation}\label{nu_i}
\nu_i(E)=\iint_E |y|^{i-2}\,dx\,dy.
\end{equation}

{Note that $\nu_i(E)$ is invariant under shifts parallel to the $x$-axis. This allows us to associate with each $D_1$ and $D_2$  their own Cartesian systems. In both systems $l$ is the $x$-axis, but in the coordinate system associated with $D_1$ the origin is at $p_1$, while in the system associated with $D_2$ the origin is at $p_2$.}

Our goal is to estimate $\nu_i(E_j)$. Fix the Cartesian system associated with $D_1$, with $p_1$ being the origin. For a point $(x,y)\in \mathcal{N}(X_0) \cup \mathcal{N}(Y_0)$ we will introduce new coordinates $(r,\theta)$ as follows. Let  $\theta=\angle(l_{\theta,1},l)$, where $l_{\theta,1}$ is the line passing through $(x,y)$ and supporting $D_1$.    Define $r$ to be the signed distance between $(x,y)$ and the foot of the perpendicular   from the point $(0,1)$ to the line $l_{\theta,1}$. (The word ``signed" means that $r>0$ in the neighborhood of $X_0$ and $r< 0$ in the neighborhood of $Y_0$).  Let $h_{D_1}(\theta)$ be the support function of $D_1$ measured from the point $(0,1)$ in the direction of $(\sin\theta,-\cos\theta)$. Using that
$$(x,y)= h_{D_1}(0) \cdot (0,1) + r (\cos\theta, \sin\theta) + h_{D_1}(\theta)\cdot (\sin\theta,-\cos\theta),$$ we will write the integral (\ref{nu_i}) in the $(r,\theta)$-coordinates associated with $D_1$.  Since the Jacobian is $|r-h'_{D_1}(\theta)|$, and $r=h'_{D_1}(\theta)$ corresponds to the point of contact of $l_{\theta,1}$ and $D_1$, we get
\begin{eqnarray*}
&&\nu_i(E_j)=\iint_{E_j} |y|^{i-2}\,dx\,dy\\
&& =\int_{\theta_{j+1}}^{\theta_j}\left| \int_{\rho_{K,D_1}(u,\xi)-h'_{D_1}(\theta)}^{\rho_{L,D_1}(u,\xi)-h'_{D_1}(\theta)}|r\sin\theta+h_{D_1}(0)-h_{D_1}(\theta)\cos\theta|^{i-2}|r-h'_{D_1}(\theta)|\,dr\right| \, d\theta\\
&& =\int_{\theta_{j+1}}^{\theta_j}\left| \int_{\rho_{K,D_1}(u,\xi)}^{\rho_{L,D_1}(u,\xi)}|r\sin\theta+h'_{D_1}(\theta)\sin\theta+h_{D_1}(0)-h_{D_1}(\theta)\cos\theta|^{i-2}r\,dr\right| \, d\theta,
\end{eqnarray*}
where   $u=(\cos\theta, \sin \theta)$, and $\xi=(\sin\theta,-\cos\theta)$. Here the absolute value of the integral with respect to $r$ is needed, since we do not know which of $\rho_K$ or $\rho_L$ is greater.

{For small $\theta$, Lemma \ref{l2.2} yields that
$$
  h'_{D_1}(\theta)\sin\theta+h_{D_1}(0)-h_{D_1}(\theta)\cos\theta \approx \sin^2 \theta.
$$
Since $E_j$ is inside $\mathcal{N}(X_0)$, there exists a constant $C>0$ such that
$$
(1-C\sin \theta)r\sin\theta\le r\sin\theta+h'_{D_1}(\theta)\sin\theta+h_{D_1}(0)-h_{D_1}(\theta)\cos\theta\le (1+C\sin \theta)r\sin\theta,
$$
where we assume that $\theta$ is small enough so that $1-C\sin \theta>0$.
}

If $i\ge 2$, for small $\theta >0$ we have
\begin{multline*}
\left(\frac{1-C\sin\theta}{1+C\sin\theta}\right)^{ i-2 } (r\sin\theta)^{i-2}\le \left( {1-C\sin\theta} \right)^{ i-2 } (r\sin\theta)^{i-2}\\
\le |r\sin\theta+h'_{D_1}(\theta)\sin\theta+h_{D_1}(0)-h_{D_1}(\theta)\cos\theta|^{i-2}\\
\le \left( {1+C\sin\theta} \right)^{ i-2 } (r\sin\theta)^{i-2}\le \left(\frac{1+C\sin\theta}{1-C\sin\theta}\right)^{ i-2 } (r\sin\theta)^{i-2}.
\end{multline*}
On the other hand, for $i<2$,
\begin{multline*}
\left(\frac{1+C\sin\theta}{1-C\sin\theta}\right)^{ i-2 } (r\sin\theta)^{i-2}\le \left( {1+C\sin\theta} \right)^{ i-2 } (r\sin\theta)^{i-2}\\
\le |r\sin\theta+h'_{D_1}(\theta)\sin\theta+h_{D_1}(0)-h_{D_1}(\theta)\cos\theta|^{i-2}\\
\le \left( {1-C\sin\theta} \right)^{ i-2 } (r\sin\theta)^{i-2}\le \left(\frac{1-C\sin\theta}{1+C\sin\theta}\right)^{ i-2 } (r\sin\theta)^{i-2}.
\end{multline*}

Thus, for both $i\ge 2$ and $i<2$, we have
\begin{multline}\label{e2.2}
\frac{1}{i}\int_{\theta_{j+1}}^{\theta_j} \left(\frac{1-C\sin\theta}{1+C\sin\theta}\right)^{|i-2|}(\sin\theta)^{i-2}\left| {\rho_{K,D_1}^i(u,\xi)}-  \rho_{L,D_1}^i(u,\xi) \right|\, d\theta\leq \nu_i(E_j)\\
 \leq \frac{1}{i}\int_{\theta_{j+1}}^{\theta_j} \left(\frac{1+C\sin\theta}{1-C\sin\theta}\right)^{|i-2|}(\sin\theta)^{i-2}\left| {\rho_{K,D_1}^i(u,\xi)}-  \rho_{L,D_1}^i(u,\xi)\right|\, d\theta.
\end{multline}

Now apply the same estimates to $\nu_i(\varphi_1(E_j))$. Since $\varphi_1(E_j)\subset \mathcal{N}(Y_0)$, and assuming that the constant $C$ chosen above works for both $\mathcal{N}(X_0)$ and $\mathcal{N}(Y_0)$, we get

\begin{eqnarray*}
\nu_i(\varphi_1(E_j))&\geq& \frac{1}{i}\int_{\theta_{j+1}}^{\theta_j} \left(\frac{1-C\sin\theta}{1+C\sin\theta}\right)^{|i-2|}(\sin\theta)^{i-2}\left| {\rho_{K,D_1}^i(-u,\xi)}-  \rho_{L,D_1}^i(-u,\xi) \right|\, d\theta\\
&=& \frac{1}{i}\int_{\theta_{j+1}}^{\theta_j} \left(\frac{1-C\sin\theta}{1+C\sin\theta}\right)^{|i-2|}(\sin\theta)^{i-2}\left| {\rho_{K,D_1}^i(u,\xi)}-  \rho_{L,D_1}^i(u,\xi) \right|\, d\theta\\
&=& \frac{1}{i}\int_{\theta_{j+1}}^{\theta_j} \left(\frac{1-C\sin\theta}{1+C\sin\theta}\right)^{2|i-2|}\left(\frac{1+C\sin\theta}{1-C\sin\theta}\right)^{|i-2|}(\sin\theta)^{i-2}\left| {\rho_{K,D_1}^i(u,\xi)}-  \rho_{L,D_1}^i(u,\xi) \right|\, d\theta\\
&\geq& \left(\frac{1-C\sin\theta_{j}}{1+C\sin\theta_{j}}\right)^{2|i-2|}\nu_i(E_j),
\end{eqnarray*}
since $\displaystyle \frac{1-C\sin \theta }{1+C\sin \theta }$ is decreasing.

Define another sequence of angles $ \eta_j= \angle(\overrightarrow{ \varphi_1(Q_j)Q_{j+1}},\overrightarrow{p_1 X_0})$. Then calculations similar to those above give
$$
\nu_i(E_{j+1})\geq \left(\frac{1-C\sin\eta_{j}}{1+C\sin\eta_{j}}\right)^{2|i-2|}\nu_i(\varphi_1(E_j)).
$$
Thus,
\begin{equation}\label{(j+1)via(j)}
\nu_i(E_{j+1})\geq \left(\frac{1-C\sin\eta_{j}}{1+C\sin\eta_{j}}\right)^{2|i-2|} \left(\frac{1-C\sin\theta_{j}}{1+C\sin\theta_{j}}\right)^{2|i-2|} \nu_i( E_j).
\end{equation}
Observe that (\ref{abcd}) implies, for all $j$,
$$
\frac{\sin \theta_{j+1}}{\sin \theta_{j}}=\frac{\sin \theta_{j+1}}{\sin \eta_{j}}\frac{\sin \eta_{j}}{\sin \theta_{j}} \leq \frac{db}{ac}<1,
$$
and, similarly,
$$
\frac{\sin \eta_{j+1}}{\sin \eta_{j}} \leq \frac{db}{ac}.
$$
Set $\displaystyle k=\frac{db}{ac}$. Then  $\sin \theta_{j}\leq k^j \sin \theta_0\leq k^j$ and $\sin \eta_{j}\leq k^j \sin \eta_0\leq k^j$.

For sufficiently small $x>0$, we have the following inequalities: $1+x\leq e^x$ and $1-x\geq e^{-2x}$. Let  $N>0$ be large enough so that $x = C k^j$ satisfies the latter two inequalities for all $j\ge N$.
Then for all $j\ge N$, we have
$$
\nu_i(E_{j+1})\geq \left(\frac{1-Ck^j}{1+Ck^j}\right)^{4|i-2|}\nu_i(E_j)\geq \left(\frac{e^{-2Ck^j}}{e^{Ck^j}}\right)^{4|i-2|}\nu_i(E_j) =   e^{-12C|i-2|k^j}  \nu_i(E_j) .
$$
Using the latter estimate inductively, we get
$$
\nu_i(E_{j+1})\ge \prod_{m=N}^{j}e^{-12C|i-2|k^m} \nu_i(E_N)=\exp\left\{-12C|i-2|\sum_{m=N}^j k^{m}\right\} \nu_i(E_N) \ge \gamma \nu_i(E_N),
$$
where $$\gamma = \exp\left\{-12C|i-2|\sum_{m=N}^\infty  k^{m}\right\} >0.$$
Since all $E_j$ are disjoint, and since $\nu_i(E_N) \ge \tilde C \nu_i(E_0)>0$, for some constant $\tilde C$ (by virtue of (\ref{(j+1)via(j)})), we conclude that
$$
\nu_i\left(\bigcup_{j=N+1}^\infty E_j\right)=\sum_{j=N+1}^\infty \nu_i(E_j)\geq \gamma \sum_{j=N+1}^\infty\nu_i(E_N)=\infty.
$$

Since $l\cap (K\triangle L)=\{X_0,Y_0\}$, there exists a triangle $T$ with one vertex at $X_0$ satisfying $T\cap l=X_0$ and $\cup_{j=N+1}^\infty E_j\subset T$, implying
$$
\nu_i(T)\geq \nu_i\left(\bigcup_{j=N+1}^\infty E_j\right)=\infty.
$$
{However, by   \cite[Lemma 5.2.4]{Ga}, any triangle of the form $T=\{(x,y):a|x-x_0|\leq y\leq b\}$, for $a,b>0$, has finite $\nu_i$-measure. We get a contradiction.}
Thus, $\partial K = \partial L$ in $\mathcal{N}(X_0) \cap l^+$.

{\it Step 4}.   To finish the proof, we take any point $A\in \partial K$. Applying $\varphi_1$ to $A$ finitely many times, we can get a point $A'$ in $l^- \cap \partial K$. As in Step 2, produce a sequence of points $A_{j+1}=\varphi^{-1}_2(\varphi_1(A_j))$ with $A_0=\varphi_2^{-1}(A')$. As we have seen above, there is a number $M$ large enough such that $A_M\in \mathcal{N}(X_0)\cap l^+$. Applying the conclusion of Step 3, we get $A_M\in \partial K\cap \partial L$. Tracing the sequence $\{A_i\}$ backwards, we conclude that $A\in \partial K\cap \partial L$. Therefore, $K=L$.

 We now briefly comment on how to proceed if we use condition (II) of the theorem. Note that here we require that there is a point $Q\in \partial K\cap \partial L$. We  define $\varphi_1$ and $\varphi_2$ in a similar way as above, with the only difference that
$$
|QT|^i-|\varphi_j(Q)T|^i =\rho_{K,D_j}^i(u,\xi)-\rho_{K,D_j}^i(-u,\xi),
$$
for $j=1,2$. The rest of the proof goes without any changes.
}

\end{proof}

{
\begin{rem}\label{Remark_smoothness}
The $C^2$-smoothness assumption for the support functions of the bodies  $D_1$ and $D_2$ can be relaxed. As we saw above, we only need the $C^2$ condition in some neighborhoods of the points $p_1$ and $p_2$ correspondingly. Moreover, $D_1$ or $D_2$ can also be polygons.  In the latter case, $\rho_{K,D_j}$ is not well defined for finitely many supporting lines, but this is not an issue.
 Step 1 of the proof does not need any changes, since it was proved for bodies that are not necessarily strictly convex. In Step 2, we consider small one-sided neighborhoods of $X_0$ and $Y_0$, where $\rho_{K,D_j}$ is well-defined. As for Step 3, the proof will be similar to \cite[Section 5.2]{Ga}, since all supporting lines to a polygon $D_j$ passing through points $X\in \mathcal{N}(X_0)\cap l^+$ will contain the same vertex of $D_j$. Thus, as in \cite{Ga}, the measure $\nu_i$ would be invariant under $\varphi_j$. So, whenever we speak about admissible bodies, one can consider a larger class of admissible bodies by including the bodies described in this remark.
\end{rem}
}

The following is an immediate corollary of Theorem \ref{main}.

\begin{cor}\label{c2.5}
Let $K$ and $L$ be origin-symmetric convex bodies in $\mathbb R^2$ and let $D$ be a    convex body  in the interior of $K\cap L$, { such that $D$ and $-D$ are admissible bodies}. If the chords $K\cap H$ and $L\cap H$ have equal length for all $H$ supporting   $D$, then $K=L$. In particular, $D$ can be a disk {not centered at the origin.}
\end{cor}

Using the same ideas, one can prove the following.
\begin{cor}\label{c2.6}
Let $K$ and $L$ be origin-symmetric convex bodies in $\mathbb R^2$ and let $D$ be a   convex body  outside of $K\cup L$ (either a  polygon or a  body  with a $C^2$ support function). If the chords $K\cap H$ and $L\cap H$ have equal length for all $H$ supporting   $D$, then $K=L$.
\end{cor}

If $H$ is a supporting line to a body $D\subset\mathbb R^2$, we will denote by $H^+$ the half-plane bounded by $H$ and disjoint from the interior of $D$.

\begin{thm}\label{t2.7}
Let $K$ and $L$ be convex bodies in $\mathbb R^2$ and let $D_1$ and $D_2$ be two admissible convex bodies (either convex polygons or bodies with $C^2$ support functions) in the interior of $K\cap L$. If $ \textnormal{vol}_2(K\cap H^+) = \textnormal{vol}_2(L\cap H^+)$ for every $H$ supporting $D_1$ or $D_2$, then $K=L$.
\end{thm}

\begin{proof}

{First we will prove the following claim. Let $K$ an $L$ be convex bodies in $\mathbb R^2$, $D$ be a convex body in the interior of $K\cap L$, where $D$ is either a body with $C^2$ support function or a polygon.
If $\mbox{vol}_2(K\cap H^+) = \mbox{vol}_2(L\cap H^+)$ for every $H$ supporting $D$, then $$\rho^2_{K,D}(u,\xi)-\rho^2_{K,D}(-u,\xi)=\rho^2_{L,D}(u,\xi)-\rho^2_{L,D}(-u,\xi),$$ for every $\xi \in S^1$ and $u\in S^1\cap \xi^\perp$, whenever well-defined.
(Note that in the case when $D$ is a polygon,  the radial functions above are not well-defined for finitely many directions $\xi$ that are orthogonal to the edges of $D$).

\begin{center}
\includegraphics[width=0.7\linewidth]{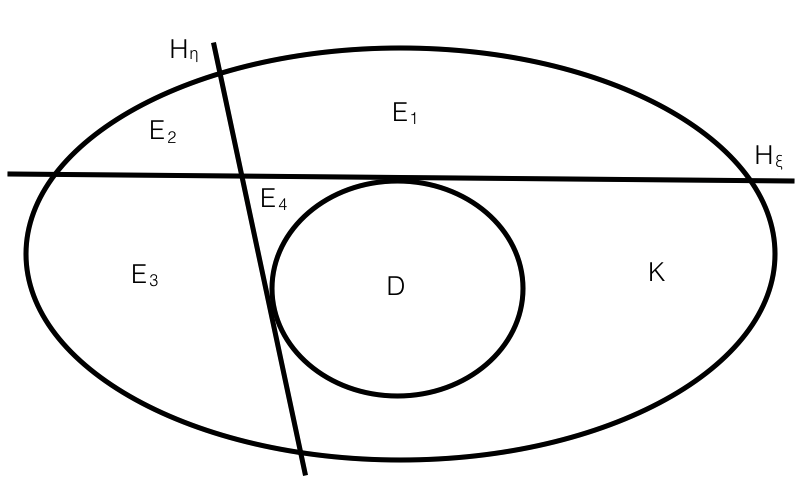}\\
Figure 4
\end{center}

We will treat simultaneously both the case of smooth bodies and polygons.
To prove the claim, let $\xi$ be any unit vector  (and $\xi$ is not orthogonal to an edge of $D$, if $D$ is a polygon).  Let $H_\xi$ be the supporting line orthogonal to $\xi$.
Let  $\zeta\in S^1\cap \xi^\perp$. For a small angle $\phi>0$ let $\eta = \cos\phi\, \xi + \sin\phi\, \zeta$, and denote by $H_\eta$ the supporting line orthogonal to $\eta$. Define the following sets: $E_1 =H_\xi^+\backslash H_\eta^+$, $E_2 =H_\xi^+\cap H_\eta^+$, $E_3 = H_\eta^+\backslash H_\xi^+$, and $E_4$ is the curvilinear triangle enclosed by $H_\xi$, $H_\eta$, and the boundary of $D$; see Figure 4.

Note that when $\eta$ and $\xi$ are close enough, we have $E_4\subset K\cap L$, and   $E_4$ is empty if $D$ is a polygon.  By the assumption of the theorem,
$$
\mbox{vol}_2((E_1\cup E_2)\cap K)-\mbox{vol}_2((E_3\cup E_2)\cap K)= \mbox{vol}_2((E_1\cup E_2)\cap L)-\mbox{vol}_2((E_3\cup E_2)\cap L),
$$
implying
\begin{equation}\label{E5}
\mbox{vol}_2((E_1\cup E_4)\cap K)-\mbox{vol}_2((E_3\cup E_4)\cap K)= \mbox{vol}_2((E_1\cup E_4)\cap L)-\mbox{vol}_2((E_3\cup E_4)\cap L).
\end{equation}
Now we will consider the following coordinate system $(r,\theta)$
associated with $D$. For a point $(x,y)$ outside of $D$, we  let
$(x,y) = h_D(\theta) \left(\cos\theta\, \xi + \sin\theta\,
  \zeta\right) + r (\sin\theta\, \xi - \cos\theta\, \zeta)$, where
$h_D(\theta)$ is the support function of $D$ in the direction of
$v=\cos\theta\, \xi + \sin\theta\, \zeta$. Setting $w = \sin\theta\,
\xi - \cos\theta\, \zeta$, and observing that the Jacobian is $|r+h_D'(\theta)|$, we get
\begin{eqnarray*}
&&
\int_{0}^{\phi}\int_{h'_{D}(\theta)}^{\rho_{K,D}(w,v)+h'_{D}(\theta)}|r+h'_{D}(\theta)|\,dr
   \,d\theta-
   \int_{0}^{\phi}\int_{h'_{D}(\theta)}^{\rho_{K,D}(-w,v)+h'_{D}(\theta)}|r+h'_{D}(\theta)|\,dr
   \,d\theta\\
&&=\int_{0}^{\phi}\int_{h'_{D}(\theta)}^{\rho_{L,D}(w,v)+h'_{D}(\theta)}|r+h'_{D}(\theta)|\,dr \,d\theta-\int_{0}^{\phi}\int_{h'_{D}(\theta)}^{\rho_{L,D}(-w,v)+h'_{D}(\theta)}|r+h'_{D}(\theta)|\,dr \,d\theta,
\end{eqnarray*}
which after a variable change becomes
$$
\int_{0}^{\phi}\int_{0}^{\rho_{K,D}(w,v)}r\,dr \,d\theta- \int_{0}^{\phi}\int_{0}^{\rho_{K,D}(-w,v)}r\,dr \,d\theta=\int_{0}^{\phi}\int_{0}^{\rho_{L,D}(w,v)}r\,dr \,d\theta-\int_{0}^{\phi}\int_{0}^{\rho_{L,D}(-w,v)}r\,dr \,d\theta.
$$
Differentiating both sides with respect to $\phi$, and setting
$\phi=0$, we get
 $$\rho^2_{K,D}(u,\xi)-\rho^2_{K,D}(-u,\xi)=\rho^2_{L,D}(u,\xi)-\rho^2_{L,D}(-u,\xi),$$
as claimed.

To finish the proof of the theorem, note that $\partial K\cap \partial
L\cap l^-\ne \emptyset$, where $l$ is the common supporting line to
$D_1$ and $D_2$ as in Theorem \ref{p3.4}; otherwise we would have
$\mbox{vol}_2(K\cap l^-)< \mbox{vol}_2(L\cap l^-)$ or
$\mbox{vol}_2(K\cap  l^-)> \mbox{vol}_2(L\cap  l^-)$, which
contradicts the hypotheses.

Now the conclusion follows from Theorem \ref{p3.4}.
}
\end{proof}

 {
\begin{cor}
Let $K$ be a convex body in $\mathbb{R}^2$ and let $D$ be a disk in the interior of $K$. If $\textnormal{vol}_2(K\cap H^+)=$ const for every $H$ supporting $D$, then $K$ is also a disk.
\end{cor}
}
\begin{proof}
 {From the proof of Theorem \ref{t2.7} we see that the condition $\mbox{vol}_2(K\cap H^+)=C$ for every line $H$ supporting $D$ implies   $\rho^2_{K,D}(u,\xi)-\rho^2_{K,D}(-u,\xi)=0$ for all $\xi\in S^1$ and $u\in S^1\cap \xi^\perp$. Without loss of generality, let $D$ be a disk of radius $1$.  Consider the mapping $\varphi$ defined as follows.  Let $Q$ be a point outside of $D$. There are two unique supporting lines to $D$ passing through $Q$. Choose the   one that lies on the right of the disk $D$ when viewing from the point $Q$. Let $T$ be the point of contact of the chosen supporting line and the disk $D$. On this line we take a point $\varphi(Q)$, such that $T$ is the midpoint of the segment $[Q,\varphi(Q)]$.

  For a point $Q\in \partial K$ introduce the coordinates $(\theta,r)$ so that $$Q= (\cos \theta, \sin \theta)+r(
\sin \theta,-\cos \theta).$$ Then,
$$
\varphi(Q)=(\theta+2\arctan r,r).
$$
}

 {
Applying $\varphi$ to $\varphi(Q)$ and iterating this procedure, we get a   set $$E=\{((\theta+2n\arctan r)\mbox{ mod}\  2\pi),r):\ n\in \mathbb{N}\}\subset \partial K.$$ Note that all points in this set are at the same distance from the origin. If $\arctan r$ is an irrational multiple of $\pi$, $E$ is a dense subset of $\partial K$, implying that $K$ is a disk. If $\arctan r$ is a rational multiple of $\pi$,  we will argue by contradiction. Assume $K$ is not a disk. By the continuity of $\partial K$, there exists a point on the boundary of $K$ with coordinates $(\theta',r')$, such that, $\arctan r'$ is an irrational multiple of $\pi$. Contradiction.
}
\end{proof}

\section{Main results: Higher dimensional cases.}
\begin{thm}\label{t2.12}
Let $K$ and $L$ be convex bodies in $\mathbb R^n$ (where $n$ is even) and let $D$ be a cube in the interior of $K\cap L$. If $\textnormal{vol}_{n-1}(K\cap H) = \textnormal{vol}_{n-1}(L\cap H)$ for any hyperplane passing through a vertex of $D$ and an interior point of $D$, then $K=L$.
\end{thm}

 {
  For $\epsilon>0$ and $\xi\in  {S}^{n-1}$, denote by
 $$
 U_\epsilon(\xi)=\{\eta\in  {S}^{n-1}:\langle \eta, \xi\rangle > \sqrt{1-\epsilon^2}\}
 $$
 the spherical  cap centered at $\xi$, and  by
 $$
 E_\epsilon(\xi)=\{\eta\in  {S}^{n-1}:|\langle \eta, \xi\rangle| <\epsilon\}
 $$
the neighborhood of the equator ${S}^{n-1}\cap \xi^\perp$.
}

\begin{lem}\label{l2.13}
{
Let $K$ and $L$ be convex bodies in $\mathbb R^n$ (where $n$ is even) containing the origin in their interiors. Let $\xi\in S^{n-1}$ and $\epsilon>0$. If $\textnormal{vol}_{n-1}(K\cap u^{\perp}) = \textnormal{vol}_{n-1}(L\cap u^{\perp})$ for every $u \in E_\epsilon(\xi)$, then $\rho^{n-1}_K(\eta)+\rho^{n-1}_K(-\eta)=\rho^{n-1}_L(\eta)+\rho^{n-1}_L(-\eta)$ for every $\eta \in U_\epsilon(\xi)$.
}
\end{lem}

\begin{proof}

{
For every even function $\psi \in C^\infty(S^{n-1})$ with support in $U_\epsilon(\xi)\cup U_\epsilon(-\xi)$, we have
\begin{multline*}
 \int_{S^{n-1}} ( \|x \|_K^{-n+1}+\|-x\|_K^{-n+1} )\psi(x)\, dx \\
=(2\pi)^{-n}\int_{S^{n-1}} (\| x \|_K^{-n+1}+\|-x\|_K^{-n+1} )^{\wedge}(u) (\psi(x/|x|) |x|^{-1})^\wedge(u)\, du,
\end{multline*}
where we used Parseval's formula on the sphere; see \cite[Section 3.4]{K}.
}

{
Since
 $(\| x \|_K^{-n+1}+\|-x\|_K^{-n+1} )^{\wedge}(u) = 2\pi (n-1) \mbox{vol}_{n-1}(K\cap u^{\perp})$ by \cite[Lemma 3.7]{K}, the assumption of the lemma yields
$$(\| x \|_K^{-n+1}+\|-x\|_K^{-n+1} )^{\wedge}(u)=(\| x \|_L^{-n+1}+\|-x\|_L^{-n+1} )^{\wedge}(u)$$
for every $u\in E_\epsilon(\xi)$. On the other hand, by formula (3.6) from \cite{GYY} or \cite[Lemma 5.1]{NRZ0}, we see that $(\psi(x/|x|) |x|^{-1})^\wedge\Big|_{S^{n-1}}$ is supported in $E_\epsilon(\xi)$.
}

{
Therefore,
\begin{eqnarray*}
&&
 \int_{S^{n-1}} ( \|x \|_K^{-n+1}+\|-x\|_K^{-n+1} )\psi(x)\, dx  \\
&&=(2\pi)^{-n}\int_{S^{n-1}} (\| x \|_L^{-n+1}+\|-x\|_L^{-n+1} )^{\wedge}(u) (\psi(x/|x|) |x|^{-1})^\wedge(u)\, du\\
&& = \int_{S^{n-1}} (\| x \|_L^{-n+1}+\|-x\|_L^{-n+1} ) \psi(x)  \, dx.
\end{eqnarray*}
Since this true for any $\psi \in C^\infty(S^{n-1})$ with support in $U_\epsilon(\xi)\cup U_\epsilon(-\xi)$,  the conclusion follows.
}
\end{proof}

\begin{deff}\label{d2.14}
{
Let $D$ be a convex polytope and $v_k$ one of its vertices. Define $C_D(v_k)$ to be the double cone centered at $v_k$ with the property that every point in $C_D(v_k)$ lies on a line through $v_k$ that has non-empty intersection with $D\setminus \{v_k\}$.
}
\end{deff}

{Note that when $D$ is a cube, $\cup_kC_D(v_k)=\mathbb{R}^n$.
\begin{rem}\label{r2.16} For simplicity, we stated Theorem \ref{t2.12} only in the case when $D$ is a cube, but, in fact, it remains valid for a larger class of polytopes. In particular, any centrally symmetric polytope $D$ satisfying the following condition will work: $\cup_kC_D(v_k)=\mathbb{R}^n$.
\end{rem}
}
\begin{proof}[Proof of Theorem \ref{t2.12}]
We will prove the theorem for the class of polytopes described in Remark \ref{r2.16}.
{Assume that $D$ is such a polytope and its center of symmetry  is at the origin $O$.

By Lemma \ref{l2.13}, if $v_i$ is a vertex of $D$, then $$
\rho^{n-1}_{K,v_i}(\xi)+\rho^{n-1}_{K,v_i}(-\xi)=\rho^{n-1}_{L,v_i}(\xi)+\rho^{n-1}_{L,v_i}(-\xi),
$$
for every $\xi\in S^{n-1}\cap (C_D(v_i) - v_i)$. Here,   $\rho_{K,v_i} $ and $\rho_{L,v_i} $ are the radial functions of $K$ and $L$ with respect to the point $v_i$.

For a point $Q\in C_D(v_i)$ define a mapping $\varphi_i$ as follows. Let $\varphi_i(Q)$ be the point on the line through $Q$ and $v_i$, such that $v_i$ lies between $Q$ and $\varphi_i(Q)$, and
$$|Q v_i|^{n-1}+|\varphi_i(Q )v_i|^{n-1} = \rho^{n-1}_{K,v_i}(\xi)+\rho^{n-1}_{K,v_i}(-\xi)=\rho^{n-1}_{L,v_i}(\xi)+\rho^{n-1}_{L,v_i}(-\xi),$$
where $\xi$ is the unit vector in the direction of $\overrightarrow{v_iQ}$.
Note that the domain of $\varphi_i$ is not the entire set $C_D(v_i)$, but it will be enough that $\varphi_i$ is defined in some neighborhood of  $(K\triangle L) \cap C_D(v_i) $.

Note that $\partial K\cap \partial L\ne \emptyset$. Otherwise one of the bodies $K$ or $L$ would be strictly contained inside the other body, thus violating the condition  $\mbox{vol}_{n-1}(K\cap H) = \mbox{vol}_{n-1}(L\cap H)$ from the statement of the theorem. Consider a point $Q\in \partial K\cap \partial L$.  There exists a vertex $v_i$ of $D$, such that  $Q\in C_D(v_i)$.
Since $D$ is origin-symmetric, there is a vertex $v_j = - v_i$. Our first goal is to show that $l(v_i,v_j)\cap \partial K = l(v_i,v_j)\cap \partial L$, where $l(v_i,v_j)$ is the line through $v_i$ and $v_j$. If $Q$ belongs to this line, we are done. If not, we will argue as follows.

Since $Q\in C_D(v_i)\cap\partial K\cap \partial L$, then $\varphi_i (Q)$ is also in $C_D(v_i)\cap\partial K\cap \partial L$. Let $\{F_m\}$ be the collection of the facets of $D$ that contain the vertex $v_i$, and let $\{n_m\}$ be collection of the corresponding outward unit normal vectors. Note that the condition $Q\in C_D(v_i)$ means that either $\langle \overrightarrow{v_i Q},n_m \rangle \ge 0$ for all $m$, or $\langle \overrightarrow{v_i Q},n_m \rangle \le 0$ for all $m$. Without loss of generality we can assume that $\langle \overrightarrow{v_i Q},n_m \rangle \ge 0$ for all $m$ (otherwise, take $\varphi_i(Q)$ instead of $Q$).

We claim that $Q\in C_D(v_i)\cap C_D(v_j)$. Indeed, the outward unit normal vectors to the facets that contain $v_j$ are $\{-n_m\}$. Thus,
$$\langle \overrightarrow{v_jQ},n_m\rangle=\langle \overrightarrow{v_iQ},n_m\rangle+\langle \overrightarrow{v_jv_i},n_m\rangle=\langle \overrightarrow{v_iQ},n_m\rangle+2\langle \overrightarrow{Ov_i},n_m\rangle\geq 0.$$
Next we claim that $\varphi_j (Q)\in  C_D(v_i)\cap C_D(v_j)$. It is clear that $\varphi_j (Q)\in    C_D(v_j)$. Thus, it is enough to show that $\langle \overrightarrow{v_i\varphi_j(Q)},n_m\rangle \le 0$ for all $m$. We have
$$\overrightarrow{v_i\varphi_j(Q)} = \overrightarrow{OQ}+\overrightarrow{Q\varphi_j(Q)} - \overrightarrow{Ov_i} = \overrightarrow{OQ}+\alpha \overrightarrow{Qv_j} - \overrightarrow{Ov_i},$$
where $\displaystyle \alpha = \frac{|Q\varphi_j(Q)|}{|Qv_j|} >1.$
So, $$\overrightarrow{v_i\varphi_j(Q)}= \overrightarrow{OQ}+\alpha \overrightarrow{Ov_j}-\alpha \overrightarrow{OQ} - \overrightarrow{Ov_i}=(1-\alpha) \overrightarrow{OQ} -(1+\alpha) \overrightarrow{Ov_i} =
(1-\alpha) \overrightarrow{v_i Q} - 2\alpha   \overrightarrow{Ov_i}.$$
Thus, for every $m$,
$$\langle \overrightarrow{v_i\varphi_j(Q)},n_m\rangle = (1-\alpha)\langle \overrightarrow{v_i Q},n_m\rangle  - 2\alpha \langle  \overrightarrow{Ov_i},n_m\rangle \le 0.$$
In a similar fashion one can show that
$\varphi_i(\varphi_j (Q))\in  C_D(v_i)\cap C_D(v_j)$. Thus we can produce a sequence of points $\{Q_k\}_{k=0}^\infty$, where $Q_0=Q$ and $Q_k=\varphi_i(\varphi_j (Q_{k-1}))$, and such that $Q_k\in C_D(v_i)\cap C_D(v_j)\cap \partial K\cap \partial L$ for all $k\ge 0$. Moreover, all these points belong to the 2-dimensional plane spanned by the points $Q$, $v_i$, and $v_j$.
As in Proposition \ref{p3.4} we have the corresponding sequence of angles $\theta_k = \angle(\overrightarrow{v_iQ_k},\overrightarrow{v_iv_j})$, with $\theta_k < \theta_{k-1}$. One can see that $\lim_{k\to \infty}\theta_k = 0$.
Since $Q_k\in   \partial K\cap \partial L$ for all $k$, we have proved that $l(v_i,v_j)\cap \partial K = l(v_i,v_j)\cap \partial L$.

Denote the points of intersection of the latter line with the boundaries of $K$ and $L$ by $X_0$ and $Y_0$, and consider any 2-dimensional plane $H$ through $X_0$ and $Y_0$. Using \cite[Lemma 7]{F}, we see that there are neighborhoods  $\mathcal{N}(X_0)$ and $ \mathcal{N}(Y_0)$ of $X_0$ and $Y_0$ correspondingly, such that
$$
H\cap\mathcal{N}(X_0)\cap \partial K=H\cap\mathcal{N}(X_0)\cap \partial L, \mbox{ and } H\cap\mathcal{N}(Y_0)\cap \partial K=H\cap\mathcal{N}(Y_0)\cap \partial L.
$$
If $P$ is a point in $C_D(v_i)\cap H$ that does not belong to $\mathcal{N}(X_0)$ or $\mathcal{N}(Y_0)$, then we apply $\varphi_j$ and $\varphi_i$ to produce a sequence of points $P_k$, which after finitely many steps will belong to $\mathcal{N}(X_0)$ or $ \mathcal{N}(Y_0)$. Thus, $P_N\in \partial K\cap \partial L$ for some large $N$. Applying inverse maps $\varphi_i^{-1}$ and $\varphi_j^{-1}$, we conclude that $P\in \partial K\cap \partial L$. Thus, we have shown that
$$
H\cap C_D(v_i)\cap \partial K=H\cap C_D(v_i)\cap \partial L.
$$
Since this is true for every $H$, we have $
  C_D(v_i)\cap \partial K=  C_D(v_i)\cap \partial L.
$

 Now consider any other vertex of $D$, say $v_m$, that is connected to $v_i$ by an edge. One can see that
$$
C_D(v_i)\cap C_D(v_m)\cap \partial K\cap \partial L\ne\emptyset.
$$
Repeating the same process as above, we   get
$$
C_D(v_m)\cap \partial K=C_D(v_m)\cap \partial L.
$$
Since we can do  this for every vertex, it follows that  $C_D(v_k)\cap \partial K=C_D(v_k)\cap \partial L$ for every $k$, and thus $K=L$.
}
\end{proof}

\begin{rem}\label{r2.15}
How to prove this in odd dimensions? Is there a different condition that guarantees a positive answer in odd dimensions? If we replace the equality of sections by the equality of derivatives of the parallel section functions, then, for example, in $\mathbb R^3$ first derivatives are not enough; cf. \cite[Remark 1]{KS}.
\end{rem}

{The next theorem is an analogue of Groemer's result for half-sections. The difference is that we look at half-sections that do not pass through the origin. We will adopt the following notation. For a   point $p\in \mathbb R^n$ and a vector $v\in S^{n-1}$, define $v_p^\perp = \{x\in \mathbb R^n: \langle x-p, v\rangle = 0\}$ and   $v_p^+ = \{x\in \mathbb R^n: \langle x-p, v\rangle \ge 0\}$.}
\begin{thm}\label{t2.16}
Let $K$ and $L$ be convex bodies in $\mathbb R^n$, $n\ge 3$, that contain a  strictly convex body $D$ in their interiors. Assume that $$\textnormal{vol}_{n-1}(K\cap H\cap v_p^+) = \textnormal{vol}_{n-1}(L\cap H\cap v_p^+),$$ for every hyperplane $H$ supporting $D$ and every unit vector $v\in H-p$, where $p=D\cap H$. Then $K=L$.
\end{thm}

\begin{proof}
Let us fix a supporting plane $H$ and consider the equality
$$
\mbox{vol}_{n-1}(K\cap H\cap v_p^+) = \mbox{vol}_{n-1}(L\cap H\cap v_p^+),
$$
for every unit vector $v\in H-p$.  Then \cite{Gr} implies that
$$
\rho^{n-1}_{K,p}(u)-\rho^{n-1}_{K,p}(-u)=\rho^{n-1}_{L,p}(u)-\rho^{n-1}_{L,p}(-u),
$$
for every  vector $u\in S^{n-1}\cap(H-p)$, where $p=D\cap H$.

Now observe that $\partial K\cap \partial L\ne \emptyset$; otherwise the condition $\mbox{vol}_{n-1}(K\cap H\cap v_p^+) = \mbox{vol}_{n-1}(L\cap H\cap v_p^+)$ would be violated.
{Moreover, if $Q\in\partial K\cap \partial L$, then by  \cite[Lemma 3]{BL}  there exists a neighborhood $\mathcal{N}(Q)$ of $Q$, such that  $\mathcal{N}(Q)\cap \partial K \subset \partial K\cap \partial L$.
Hence, $\partial K\cap \partial L$ is open in $\partial K$. On the other hand, by the continuity of the boundaries of $K$ and $L$, $\partial K\cap \partial L$ is closed in $\partial K$. Therefore,
}
$$
\partial K\cap \partial L=\partial K= \partial L.
$$
\end{proof}

\begin{cor}\label{c2.17}
Let $K$ be a convex body in $\mathbb R^n$, $n\ge 3$, that contains a ball $D$ of radius $t$ in its interior.  If $$\mathrm{vol}_{n-1}(K\cap \{\xi^\perp+t\xi\}\cap v^+) = \mbox{const},$$ for every $\xi\in S^{n-1}$ and every   vector $v\in S^{n-1}\cap  \xi^\perp$, then $K$ is a Euclidean ball.
\end{cor}

In the next theorem we will consider a different type of half-sections.
\begin{thm}\label{t2.18}
{Let $K$ and $L$ be convex bodies in $\mathbb R^n$, $n\ge 3$, that contain a ball $D$ in their interiors. Assume that
$$
\textnormal{vol}_{n-1}(K\cap H^+\cap v^\perp) =\textnormal{vol}_{n-1}(L\cap H^+ \cap v^\perp)
$$
for every hyperplane $H$ supporting $D$ and every unit vector $v\in H-p$, where $p=D\cap H$. Then $K=L$. }
\end{thm}
\begin{proof}
{
Let us fix a unit vector $v$, and consider $\xi,\zeta\in S^{n-1}\cap v^\perp$ such that $\xi\perp\zeta$. For a small $\phi$ let $\eta = \cos\phi\, \xi + \sin\phi\, \zeta$. Without loss of generality we will assume that $D$ has radius 1 and is centered at the origin. Consider the affine hyperplanes $H_\xi = \{x\in\mathbb R^n: \langle x,\xi\rangle = 1\}$ and $H_\eta = \{x\in\mathbb R^n: \langle x,\eta\rangle = 1\}$.   Let the $(n-3)$-dimensional subspace $W$ be the orthogonal compliment of span$\{\xi,\zeta\}$ in $v^\perp$. Consider the orthogonal projection of the convex body $K\cap v^\perp$ onto the 2-dimensional subspace spanned by $\xi$ and $\zeta$. The picture is identical to Figure 4, with $E_1$, $E_2$, $E_3$, and $E_4$ defined similarly. If $n=3$, we repeat the argument from the proof of Theorem \ref{t2.7}. If $n\ge 4$, we will use the following modification of this argument.

Let $\bar E_i = E_i\times W$, for $i=1$, 2, 3, 4.
Then the equality
$$\mathrm{vol}_{n-1}(K\cap v^\perp \cap H_{\xi}^+) - \mathrm{vol}_{n-1}(K\cap v^\perp \cap H_{\eta}^+) =\mathrm{vol}_{n-1}(L\cap v^\perp \cap H_{\xi}^+) - \mathrm{vol}_{n-1}(L\cap v^\perp \cap H_{\eta}^+) $$
implies
\begin{multline}\label{difference}
\mbox{vol}_{n-1}(K\cap v^\perp\cap (\bar E_1\cup \bar E_4))-\mbox{vol}_{n-1}(K\cap v^\perp\cap (\bar E_3\cup \bar E_4))\\
 =\mbox{vol}_{n-1}(L\cap v^\perp\cap (\bar E_1\cup\bar E_4))-\mbox{vol}_{n-1}(L\cap v^\perp\cap (\bar E_3\cup\bar E_4)).
\end{multline}

For $x\in  \mathrm{span}\{\xi,\zeta\}$, consider the following parallel section function: $$A_{K\cap v^\perp,W} (x) = \mbox{vol}_{n-3}(K\cap v^\perp \cap \{W+x\}).$$
Then equation (\ref{difference}) and the Fubini theorem imply
$$\int_{E_1\cup E_4} A_{K\cap v^\perp,W} (x) dx - \int_{E_3\cup E_4} A_{K\cap v^\perp,W} (x) dx = \int_{E_1\cup E_4} A_{L\cap v^\perp,W} (x) dx - \int_{E_3\cup E_4} A_{L\cap v^\perp,W} (x) dx.$$
Now we will pass to new coordinates $(r,\theta)$ as in the proof of   Theorem \ref{t2.7}, by letting $x (r,\theta) = \cos\theta\, \xi + \sin\theta\, \zeta + r (\sin\theta\, \xi - \cos\theta\, \zeta)$.
Then
\begin{multline*}\int_0^\phi \int_0^\infty |r|  A_{K\cap v^\perp,W} (x(r,\theta)) dr d\theta - \int_0^\phi \int_{-\infty}^0 |r|  A_{K\cap v^\perp,W} (x(r,\theta))  dr d\theta\\
 =
\int_0^\phi \int_0^\infty |r|  A_{L\cap v^\perp,W} (x(r,\theta))  dr d\theta - \int_0^\phi \int_{-\infty}^0 |r|  A_{L\cap v^\perp,W} (x(r,\theta))  dr d\theta .
\end{multline*}
Differentiating with respect to $\phi$ and letting $\phi = 0$, we get
\begin{equation}\label{intA}
\int_{-\infty}^\infty r  A_{K\cap v^\perp,W} (x(r,0)) dr  = \int_{-\infty}^\infty r  A_{L\cap v^\perp,W} (x(r,0)) dr.
\end{equation}
Note that
$$A_{K\cap v^\perp,W} (x(r,0)) =A_{K\cap v^\perp,W} (\xi - r\zeta) =  A_{(K-\xi)\cap v^\perp,W} (  - r\zeta)  = \int_{x\in \xi^\perp\cap v^\perp \cap \{\langle x,\zeta\rangle = -r\}} \chi (\|x\|_{K-\xi}) dx. $$
Therefore, (\ref{intA}) and the Fubini theorem give
$$\int_{\xi^\perp\cap v^\perp} \langle x,\zeta\rangle  \chi (\|x\|_{K-\xi}) dx = \int_{\xi^\perp\cap v^\perp} \langle x,\zeta\rangle  \chi (\|x\|_{L-\xi}) dx .$$
Passing to polar coordinates in $\xi^\perp\cap v^\perp$, we get
$$\int_{S^{n-1}\cap \xi^\perp\cap v^\perp} \langle w,\zeta\rangle   \|w\|_{K-\xi}^{-n+1} dw = \int_{S^{n-1}\cap \xi^\perp\cap v^\perp} \langle w,\zeta\rangle   \|w\|_{L-\xi}^{-n+1}  dw .$$
Observe, that this is true for any $\zeta\in \xi^\perp\cap v^\perp$. Furthermore, for any vector $\vartheta \in \xi^\perp$ there is a vector $\zeta\in \xi^\perp\cap v^\perp$ and a number $\beta$ such that $\vartheta = \zeta + \beta v$. Therefore,  for every $\vartheta \in \xi^\perp$ we have
$$\int_{S^{n-1}\cap \xi^\perp\cap v^\perp} \langle w,\vartheta\rangle   \|w\|_{K-\xi}^{-n+1} dw = \int_{S^{n-1}\cap \xi^\perp\cap v^\perp} \langle w,\vartheta\rangle   \|w\|_{L-\xi}^{-n+1}  dw .$$
Fixing $\xi$ and $\vartheta$, and looking at all $v\in S^{n-1}\cap \xi^\perp$, we can consider the latter equality as the equality of the spherical Radon transforms on $S^{n-1}\cap \xi^\perp$. Since the spherical Radon transform only allows to reconstruct even parts, we get
$$  \langle w,\vartheta\rangle   \|w\|_{K-\xi}^{-n+1} +  \langle - w,\vartheta\rangle   \|-w\|_{K-\xi}^{-n+1} = \langle w,\vartheta\rangle   \|w\|_{L-\xi}^{-n+1} +  \langle - w,\vartheta\rangle   \|-w\|_{L-\xi}^{-n+1} ,$$
for all $w, \vartheta \in  S^{n-1}\cap \xi^\perp$. That is,
$$\|w\|_{K-\xi}^{-n+1} -   \|-w\|_{K-\xi}^{-n+1} =     \|w\|_{L-\xi}^{-n+1} - \|-w\|_{L-\xi}^{-n+1} , \mbox{ for all } w\in S^{n-1}\cap \xi^\perp.$$
We finish the proof as in Theorem \ref{t2.16}.
}
\end{proof}

Below we will prove an analogue of the result of Falconer and Gardner for halfspaces.  We will need the following lemma.

{
\begin{lem}\label{l3.9}
Suppose $i>0$. Let $K$ and $L$ be convex bodies in $\mathbb R^n$, $p_1$ and $p_2$ be distinct points in the interior of $K\cap L$, and $l$ be the line passing through $p_1$ and $p_2$. If for all $\xi\in {S}^{n-1}$,
 \begin{equation}\label{condition_l2.9}
\rho_{K,p_j}^{i}(\xi) - \rho_{K,p_j}^{i}(-\xi)=\rho_{L,p_j}^{i}(\xi) - \rho_{L,p_j}^{i}(-\xi),\mbox{ for } j=1,2,
 \end{equation} and $\partial K\cap \partial L \ne\emptyset,$ then $K=L$.
\end{lem}}
\begin{proof}
Our first goal is to prove that $\partial K\cap l = \partial L \cap l$. Let $Q_0 \in \partial K\cap \partial L$. If $Q_0\in l$, we are done. Otherwise, we define two maps $\varphi_1,\varphi_2$ as follows. If $Q$ is a point distinct from $p_1$, then $\varphi_1(Q)$ is defined to be the point on the line passing through $Q$ and $p_1$,  such that $p_1$ lies between $Q$ and $\varphi_1(Q)$ and
  $$
|Qp_1|^{i}-|p_1\varphi_1(Q)|^{i}=\rho_{K,p_1}^{i}(\xi) - \rho_{K,p_1}^{i}(-\xi),
$$
where $\displaystyle \xi = \frac{\overrightarrow{p_1Q}}{|p_1Q|}$.

Note that the domain of $\varphi_1$ contains the set $K\triangle L$. The map $\varphi_2$ is defined similarly with $p_1$ replaced by $p_2$.

 For the chosen point $Q_0\in \partial K\cap \partial L $ consider the $2$-dimensional plane $H$ passing through $Q_0$, $p_1$, and $p_2$.   Construct a sequence of points $\{Q_j\}\subset \partial K\cap \partial L\cap H$, satisfying  $Q_{j+1}=\varphi^{-1}_2(\varphi_1(Q_j))$, and a sequence of angles $\{\theta_j\}=\{\angle(\overrightarrow{Q_j\varphi_1(Q_j)},l)\}$. One can see that $\lim_{j\to \infty}\theta_j =0$, and therefore the limit
$$
X_0=\lim_{j\to\infty}Q_j
$$
is a point on $l\cap\partial K\cap\partial L$. The claim that $\partial K\cap l = \partial L \cap l$ is now proved.

Let $V$ be any $2$-dimensional affine subspace of $\mathbb R^n$ that  contains the line $l$. Consider the bodies $K\cap V$ and $L\cap V$ in $V$. The line $l$ cuts both these bodies in two parts, $K \cap V= K_1 \cup K_2$ and $L\cap V = L_1\cup L_2$, so that $K_1$ and $L_1$ are on the same side of $l$. Since $K\cap l= L\cap l$, the following star bodies are well-defined: $\tilde K = K_1\cup L_2$ and $\tilde L = K_2\cup L_1$.  Condition (\ref{condition_l2.9}) now implies
$$
\rho_{\tilde{K},p_j}^{i}(\xi) + \rho_{\tilde{K},p_j}^{i}(-\xi)=\rho_{\tilde{L},p_j}^{i}(\xi) + \rho_{\tilde{L},p_j}^{i}(-\xi),\mbox{ for } j=1,2.
$$
Now we can use \cite[Theorem~6.2.3]{Ga}  to show that $\tilde K = \tilde L$, implying that $K\cap V= L\cap V$. Since $V$ was an arbitrary affine subspace containing $l$, it follows that $K=L$.
\end{proof}
{
\begin{rem}  A version of this lemma  for a smaller set of values of $i$ (but without the assumption $\partial K\cap \partial L\ne\emptyset$) was proved by Koldobsky and Shane, \cite[Lemma 6]{KS}.
\end{rem}
}

With the help of Lemma \ref{l3.9} we obtain the following result.

\begin{thm}\label{t2.19}
Let   $K$ and $L$ be convex bodies in $\mathbb R^n$ containing two distinct points $p_1$ and $p_2$ in their interiors. If for  every $v\in S^{n-1}$, we have $$ \textnormal{vol}_{n}(K\cap v^+_{p_j}) =  \textnormal{vol}_{n}(L\cap   v^+_{p_j})\mbox{ for }j=1,2,$$ then $K=L$.
\end{thm}
\begin{proof} By   \cite{Gr}, we have $\rho_{K,p_j}^{n}(\xi)-\rho_{K,p_j}^{n}(-\xi)=\rho_{L,p_j}^{n}(\xi)-\rho_{L,p_j}^{n}(-\xi)$, for $j=1,2$, and every $\xi\in S^{n-1}$.
Also  observe that $ \partial K\cap \partial L \ne\emptyset$. Otherwise one of $K$ or $L$ would be strictly contained inside the other, which would contradict  the hypothesis of the theorem.
Now the result follows from Lemma \ref{l3.9}.
\end{proof}

Note that Problem \ref{BL1} is open even in the case of bodies of revolution when the center of the ball lies on the axis of revolution. However, if we consider a ball that does not intersect the axis of revolution, then the problem has a positive answer.

\begin{thm}\label{t2.20}
Let   $K$ and $L$ be convex bodies of revolution in $\mathbb R^n$ with the same axis of revolution. Let $D$ be a convex body in the interior of both $K$ and $L$ such that $D$ does not intersect the axis of revolution. If for every hyperplane $H$ supporting $D$ we have $$ \textnormal{vol}_{n-1}(K\cap H) = \textnormal{vol}_{n-1}(L\cap H  ),$$ then $K=L$.
\end{thm}
\begin{proof}
Consider the two supporting hyperplanes of $D$ that are perpendicular to the axis of revolution. Let $p$ and $q$ be the points where these hyperplanes intersect the axis of revolution.

Note that every plane passing through $p$ (or $q$) can be rotated around the axis of revolution until it touches the body $D$. Due to the rotational symmetry of the bodies $K$ and $L$ we obtain that
 $$\mbox{vol}_{n-1}(K\cap (p+\xi^{\perp})) = \mbox{vol}_{n-1}(L\cap (p+\xi^{\perp})  ),$$ and $$\mbox{vol}_{n-1}(K\cap (q+\xi^{\perp})) = \mbox{vol}_{n-1}(L\cap (q+\xi^{\perp})  ),$$
for every $\xi\in S^{n-1}$.

The conclusion now follows from the corresponding result of  Falconer \cite{F} and Gardner \cite{Ga}, described in the introduction.
\end{proof}

\end{document}